\documentclass[12pt]{article}
\usepackage{amsmath,amsfonts,amssymb,amsthm}
\newtheorem{theorem}{Theorem}
\newtheorem{remark}{Remark}

\newtheorem{prop}{Proposition}

\newcommand{\D}{\mathbb{D}}
\newcommand{\E}{\mathsf{E}}
\newcommand{\Prob}{\mathsf{P}}

\begin{document}
\begin{center}
{\Large\bf Generalized fractional calculus and some models of generalized counting processes}
 \end{center}

\begin{center}
{{\bf Khrystyna Buchak}}
\end{center}
\begin{center}
 {\it Faculty of Mathematics and Digital Technologies\\ 
	Uzhgorod National University,  \\Narodna Square 3,
	88000 Uzhgorod, Ukraine}
\end{center}

\begin{center}
{{\bf Lyudmyla Sakhno}}
\end{center}
\begin{center}
 {\it Faculty of Mechanics and Mathematics\\ Taras Shevchenko National University of Kyiv,\\ Volodymyrska St. 64,
	01601 Kyiv, Ukraine}
\end{center}

\begin{center}
{\small khrystyna.buchak@uzhnu.edu.ua (K.~Buchak),\\lyudmyla.sakhno@knu.ua (L.~Sakhno)} 
\end{center}

\medskip

\begin{abstract}
{In the paper we consider models of generalized counting processes time-changed by a general inverse subordinator, we characterize their distributions and present  governing equations for them. The equations are given in terms of the generalized fractional derivatives, namely, convolutions-type derivatives with respect to Bern\v{s}tein functions. Some particular examples are presented.}


\end{abstract}

\bigskip

\section{Introduction}

Stochastic processes with random time-change has become a well-established and highly ramified branch of the modern theory of stochastic processes which attracts more and more attention due to various applications in financial, biological, ecological, physical, technical and other fields of research. 

Rich class of models is provided by time-changed Poisson processes, of which the most intensively studied are two fractional extensions of the Poisson process, namely, the space-fractional and time-fractional Poisson processes obtained by choosing a stable subordinator or its inverse in the role of time correspondingly. We refer, for example, to papers  \cite{GOS}, \cite{OP}, \cite{OT}, \cite{BDOV}, \cite{BO}, \cite{LST},  \cite{OP2013}, among many others (see also references therein). In particular, in paper \cite{OT} a general class of time-changed Poisson processes $N^f(t)=N(H^{f}(t))$, $t>0$,  was introduced and studied, where $
N(t)$ is a Poisson process and $H^{f}(t)$ is an arbitrary subordinator with
Laplace exponent $f,$ independent of $N(t)$, their distributional properties, hitting times and governing equations were presented (see, also \cite{GOS}). In paper \cite{B} Poisson processes time changed by general inverse subordinators were studied, the governing equations for their marginal distributions were presented and some other properties were described. Poisson process itself, being in a sense a core object concerning applicability to count data and simple tractability, however, as a reverse side of its simplicity, is a rather restrictive model. Therefore, it has been quite natural to 
%go beyond the Poissonian model and 
search for its extensions and generalizations  %in order 
to provide some new useful features and  properties needed for applications. For example, time changed processes $N(H^{f}(t))$ allow for jumps of arbitrary size and %have some 
other interesting properties (\cite{OT}, \cite{GOS}).

Recently, generalized counting processes, their time-changed versions and fractional extensions have been intensively investigated, including Poisson and Skellam processes of order $k$, P\'olya-Aeppli process of order $k$, Bell-Touchard process, Poisson-logarithmic process, etc. and their fractional extensions. We refer  to papers \cite{DiC}, \cite{KM}, \cite{CM}, \cite{KLS}, \cite{KK2022a}, \cite{KK2022b}, \cite{KK2022Arxiv}, to mention only few, see also references therein.

In the present paper we consider several particular models of generalized counting processes time-changed by a general inverse subordinator, characterize their distributions and present  governing equations for them, which are obtained following the technique  presented in \cite{B}. The equations are given in terms of the generalized Caputo-Djrbashian derivatives, which called also convolutions-type derivatives with respect to Bern\v{s}tein functions. These generalized derivatives were introduced in \cite{K} and \cite{T} and have been widely used to describe various advanced stochastic models.  

The convolution-type derivatives allow to study  the properties of subordinators and their inverses in the unifying manner (\cite{T, MT}), in particular, the governing equations for their densities can be given in terms of   convolution-type derivatives. These properties will be the key tools for our study: by considering several models of time-changed processes, we elucidate the technique leading from the equations for the densities of inverse subordinators and their Laplace transforms to the equations for probabilities of  processes time-changed by inverse subordinators and equations for some related functions. 

The paper is organized as follows. Section 2 collects the basic definitions and facts on the  convolution-type derivatives needed in our study. We also consider, as a warmup example, the time changed Poisson process $N(Y^{f}(t))$,$t>0$,  where $Y^{f}$ is the inverse to the subordinator with
Laplace exponent $f,$ independent of $N(t)$. We partly extend the results from \cite{B} for this process. This   simplest case allows to demonstrate very transparently the technique  underpinning the derivations of the corresponding results for more general models considered in the next sections. Namely,  in Sections 3 and 4 we study 
the processes $N^{\psi, f}(t) = N\left(H^\psi(Y^f(t))\right)$, $t>0$, and $M^{\psi, f}(t) = M\left(H^\psi(Y^f(t))\right)$, $t>0$, where $M$ is the generalized counting process. Time change is done by the independent subordinator $H^\psi$ and inverse subordinator $Y^f$ related to Bern\v{s}tein functions $\psi$ and $f$ correspondingly. We present the differential-difference equations for probabilities of the processes $N^{\psi, f}$ and $M^{\psi, f}$, which are given in terms of the convolutions-type derivative with respect to function $f$ and the difference operator related to function $\psi$. Expressions for the probabilities of these processes are also presented in terms of the Laplace transform of the process $Y^f$. Some other properties of these processes are described and a comparison with related results in the literature are given. As examples, we consider the processes $N^{\psi, f}$ and $M^{\psi, f}$, where $H^\psi$ is a compound Poisson-Gamma process.

\section{Preliminaries}

{\bf Generalized fractional derivatives.} 
We review the main definitions and some  facts on the generalized fractional derivatives, which will be important for our study. (For more details see, e.g., \cite{MT, T}.) 

Let $f(x)$ be a Bern\v{s}tein function:
\begin{equation}\label{fx}
f(x)=a+bx+\int_{0}^{\infty}\left(1-e^{-xs}\right)\overline{\nu}(ds), \,\,x>0, \,\,a, b\geq0, 
\end{equation}
 $\overline{\nu}(ds)$ is a non-negative measure on $(0,\infty)$ (the L\'evy measure for $f(x)$) such that
\begin{equation*}
\int_{0}^{\infty}\left(s\wedge 1\right)\overline{\nu}(ds)<\infty.
\end{equation*}
The generalized Caputo-Djrbashian (C-D) derivative, or convolution-type derivative, with respect to the Bern\v{s}tein function $f$ is defined on the space of absolutely continuous functions as follows (\cite{T}, Definition 2.4):
\begin{equation}\label{fDt}
\mathcal{D}^f_t u(t)=b \frac{d}{dt}u(t)+\int_{0}^{t}\frac{\partial }{\partial t }u(t-s)\nu (s)ds,
\end{equation}
where $\nu (s)=a+\overline{\nu}(s,\infty)$ is the tail of the L\'evy measure $\overline{\nu}(s)$ of the function~$f$.

In the case where $f(x)=x^{\alpha},x>0, \alpha\in (0,1)$, the derivative \eqref{fDt} reduces to the fractional C-D
derivative:
\begin{equation}\label{C-D}
\mathcal{D}^f_t u(t)=\frac{d^{\alpha}}{dt^{\alpha}}u(t)=\frac{1}{\Gamma(1-\alpha)}\int_{0}^{t}\frac{u^\prime(s)}{\left(t-s\right)^{\alpha}}ds.
\end{equation}

 For the Laplace transform of the derivative \eqref{fDt} the following relation holds (\cite{T}, Lemma 2.5):
\begin{equation*}\label{LfDt}
\mathcal{L}\left[\mathcal{D}^f_t u\right](s)=f(s)\mathcal{L}\left[u\right](s)-\frac{f(s)}{s}u(0), s>s_0,
\end{equation*}
for $u$ such that $|u(t)|\leq\mathsf{M} e^{s_0t}$, $M$ and $s_0$ are some constants. Similarly to the C-D fractional derivative, the convolution type derivative can be  alternatively defined via its Laplace transform.

The generalization of the classical Riemann-Liouville (R-L) fractional derivative is introduced in \cite{T} by means of another convolution-type derivative with respect to $f$ given by the following formula:
\begin{equation}\label{fDDt}
\mathbb{D}_t^f u(t)=b \frac{d}{dt}u(t)+\frac{d}{dt}\int_{0}^{t}u(t-s)\nu (s)ds.
\end{equation}
The derivatives $\mathcal{D}^f_t$ and $\mathbb{D}_t^f$ are related as follows (see, \cite{T}, Proposition 2.7):
\begin{equation*}\label{fDtrelation}
\mathbb{D}_t^f u(t)=\mathcal{D}^f_t u(t) +\nu (t)u(0).
\end{equation*}

Bern\v{s}tein functions are associated in a natural way with subordinators.

Let $H(t), t\geq0,$ be a subordinator, that is, nondecreasing L\'evy process. Its Laplace transform is of the form: 
$$\mathcal{L}[H(t)](s)=\E e^{-sH(t)}= e^{-t f(s)},$$
 where the function $f$, called the Laplace exponent, is a Bern\v{s}tein function. Consider a subordinator $H^f$ with the Laplace exponent $f$ given by \eqref{fx}, and let $Y^f$ be its inverse process defined as
\begin{equation}\label{Yf}
Y^f(t)=\inf \left\{s\geq0:H^f(s)>t\right\}.
\end{equation}
It was shown in \cite{T} that the distribution of the inverse process $Y^f$ has a density 
$\ell_f(t,x)={\rm P}\{Y^f(t)\in dx\}/dx$ 
provided that the following condition holds:

\medskip
\noindent{\bf Condition I.} $\overline{\nu}(0,\infty)=\infty$ and the tail $\nu(s)=a+\overline{\nu}(s,\infty)$ is absolutely continuous.
\medskip

 The Laplace transform of the density with respect to $t$ has the form (\cite{T}):
\begin{equation*}
\mathcal{L}_t\left(\ell_f(t,x)\right)(r)=\frac{f(r)}{r}e^{-xf(r)}.
\end{equation*}

The density $\ell_f(t,u)$ of the inverse process $Y^f$ satisfies the follo\-wing equation (\cite{T}, Theorem 4.1):
\begin{equation}\label{dlf}
{\mathbb{D}}_t^f \ell_f(t,u)=-\frac{\partial}{\partial  u} \ell_f(t,u),
\end{equation}
subject to
\begin{equation}\label{dlfin}
\ell_f(t, {u}/{b})=0,\,\, \ell_f(t,0)=\nu(t),\,\, \ell_f(0,u)=\delta(u).
\end{equation}

% The following fact is important for deriving representations of solutions to various classes of partial differential equations involving generalized fractional derivatives.
%
%\begin{proposition}{\rm(\cite{DOS})}
%\label{prUseful}
%Let $L$ be the inverse process for a subordinator with
%Bern\v{s}tein function $f$, and assume that Condition I holds.
The space Laplace transform of the density $\ell_f(t,x)$ 
\begin{equation}\label{tilde_l}
\tilde{\ell_f}(t,\lambda)=\int_{0}^{\infty}e^{-\lambda x} \ell_f(t,x)dx= {\E}e^{-\lambda L^f(t)}
\end{equation}
is an eigenfunction of the operator $\mathcal{D}^f_t$, that is, satisfies the equation
	\begin{equation}\label{lapl}
		\mathcal{D}^f_t\tilde{\ell_f}(t,\lambda)=-\lambda \tilde{\ell_f}(t,\lambda).
	\end{equation}
%\end{proposition}
(see, \cite{B, K, MT}). In the case where $f(x)=x^{\alpha},x>0, \alpha\in (0,1)$, $\tilde{\ell_f}(t,\lambda)=\mathcal{E}_\alpha(-\lambda t^\alpha)$, where $\mathcal{E}_\alpha(\cdot)$ is the Mittag-Leffler function.

Equations \eqref{dlf}--\eqref{lapl} are important for the study of the processes time changed by inverse subordinators as can be seen in what follows. We first consider the Poisson process with an inverse subordinator as a simplest example which demonstrates the technique to be applied further to more general models. 
\medskip

\noindent {\bf Poisson process time changed by an inverse subordinator.} 
Let $N(t)$ be the Poisson process with intensity $\lambda,$ and  $Y^f(t) $ be the inverse subordinator.
\begin{theorem}\label{t1}
        Let Condition I hold. Then the marginal distributions $p_k^f(t)=P\left\{N(Y^f(t))=k\right\}$, $k=0,1,\dots$ satisfy the differential equation
    \begin{equation}
        	\mathcal{D}^f_t p_k^f(t)=\lambda\left(p_k^f(t)-p_{k-1}^f(t)\right) \label{Dtfpkf}
    \end{equation}
	and
	\begin{equation}
        	p_{k}^f(t)=\frac{\left(-\lambda \partial_{\lambda}\right)^k }{k!}\tilde{\ell}_f(t,\lambda) \label{pkt10},
    \end{equation}
	where $\tilde{\ell}_f(t,\lambda)=\int_{0}^{\infty}e^{- \lambda u}\ell_f(t,u) du$ is the Laplace transform of the density of the inverse subordinator $Y^f(t).$ 
	
	The probability generating function of the process $N^{f}$ has the form
\begin{equation}\label{pgfNf}
             G^{f}(u,t)=\tilde \ell(t, \lambda(1-u)),\, |u|<1,
\end{equation}
and satisfies the equation
\begin{equation} \label{DGf}
        	\mathcal{D}^f_t G^{f}(u,t)=-\lambda(1-u)G^{f}(u,t)
    \end{equation}
    with $G^{f}(u,0)=1$.
\end{theorem}

\begin{proof}
        Proof of equation (\ref{Dtfpkf}).
    For the probabilities $p_x^f(t)$ we have:
    \begin{eqnarray}
    p_x^f(t)&=&\Prob\left\{N\left(Y^f(t)\right)=x\right\}
    =\int_{0}^{\infty}p_x(u)\ell_f(t,u)du\nonumber\\
    &=&\int_{0}^{\infty}\frac{e^{-\lambda u}(\lambda u)^x}{x!}\ell_f(t,u)du, \quad x=0,1,2,\dots
    \end{eqnarray}
    We take the generalized R-L convolution-type derivative $\mathbb{D}^f_t$ given by \eqref{fDDt} 
        and use the equation for the density $\ell^f(t,u)$ of the inverse subordinator:
    \begin{equation}
    {\mathbb{D}}^f_t \ell_f(t,u)=-\frac{\partial}{\partial  u} \ell_f(t,u),
    \end{equation}
    with
    \begin{equation}
    \ell_f(t,0)=\nu(t), \,\,\,\, \ell_f(0,u)=\delta(u).
    \end{equation}

    We obtain:
    \begin{eqnarray}
    {\mathbb{D}}^f_t p_x^f(t)
    &=&\int_{0}^{\infty}p_x(u){\mathbb{D}}^f_t\ell_f(t,u)du=- \int_{0}^{\infty}p_x(u)\frac{\partial}{\partial  u}\ell_f(t,u)du\nonumber\\
    &=&\int_{0}^{\infty}\ell_f(t,u)\frac{\partial}{\partial  u}p_x(u)du-p_x(u)\ell_f(t,u)\big|_{u=0}^\infty\nonumber\\
    &=&\int_{0}^{\infty}\ell_f(t,u)(-\lambda[p_x(u)-p_{x-1}(u)])du+p_x(0)\ell_f(t,0)\nonumber\\
    &=&-\lambda\left[p^f_x(t)-p^f_{x-1}(t)\right]+p_x(0)\nu(t). \label{11}
    \end{eqnarray}
    Using the relation 
    \begin{equation}\label{fDtrelation}
        \mathbb{D}^f_t u(t)=\mathcal{D}^f_t+\nu (t)u(0)
    \end{equation}
    between convolution derivatives of C-D and R-L types, we have:
    \begin{equation}\label{12}
    \mathcal{D}^f_t p_x^f(t) = \mathbb{D}^f_t p_x^f(t)-\nu(t)p_x^f(0),
    \end{equation}
     note also that
    \begin{equation}\label{13}
    p_x^f(0)=\int_{0}^{\infty}p_x(u)\ell_f(0,u)du=\int_{0}^{\infty}p_x(u)\delta(u)du=p_x(0)=1.
    \end{equation}
    From \eqref{11}, taking into account \eqref{12}-\eqref{13}, we finally obtain:
    \begin{equation*}
\mathcal{D}^f_t p_x^f(t)=-\lambda\left[p_x^f(t)-p_{x-1}^f(t)\right].
    \end{equation*}
    Proof of equation (\ref{pkt10}):
    \begin{eqnarray}
            \Prob\left\{N(Y^f(t))=k\right\} 
                &=&\int_{0}^{\infty}\Prob(N(s)=k)\Prob(Y^f(t)\in ds)\nonumber\\
                 &=&\int_{0}^{\infty}\frac{\left(-\lambda \partial_{\lambda}\right)^k }{k!}e^{-\lambda s}\Prob \left\{Y^f(t)\in ds\right\}\nonumber\\
                 &=&\E\left[\frac{\left(-\lambda \partial_{\lambda}\right)^k }{k!}\exp\left\{-\lambda Y^f(t)\right\}\right]\nonumber\\
                 &=&\frac{\left(-\lambda \partial_{\lambda}\right)^k }{k!}\E \exp\left(-\lambda Y^f(t)\right)=\frac{\left(-\lambda \partial_{\lambda}\right)^k }{k!}\tilde{\ell}_f(t,\lambda).\nonumber
    \end{eqnarray}
    
       We used above that the distribution of the Poisson process $N$ with the rate $\lambda$ can be writen as follows:
          \begin{eqnarray*}
              p_k(t)= \frac{\left(-\lambda \partial_{\lambda}\right)^k}{k!}e^{-\lambda t}.
          \end{eqnarray*}
         For the probability generating function we have 
       $$ %\begin{eqnarray*}
          G^f(u,t) = \E u^{N(Y^f(t))}\int_{0}^{\infty} e^{-s\lambda(1-u)}\ell_f(t,s)ds= 
           \E e^{-\lambda(1-u)Y^f(t)}=\tilde{\ell}_f(t,\lambda(1-u)),$$
        %\end{eqnarray*}
        therefore, we obtain formula \eqref{pgfNf}. Now we use the fact, that $\tilde{\ell}_f(t,\lambda)$ is an eigenfunction of the operator $\mathcal{D}^f_t$ (see \eqref{lapl}),
        from which we conclude that $G^f(u,t)$, being given by $\tilde{\ell}_f(t,\lambda(1-u))$, satisfies equation \eqref{DGf}.     
\end{proof}

\begin{remark}{\rm
Note that \eqref{lapl} can be deduced directly from equation \eqref{Dtfpkf} by taking $k=1$, since $p_0^f(t)=\Prob\left\{N\left(Y^f(t)\right)=0\right\}
    =\int_{0}^{\infty}e^{-\lambda u}\ell_f(t,u)du=\tilde \ell(t,\lambda)$.}
\end{remark}

\section{Models of time-changed Poisson processes}
 Consider the time-changed Poisson processes $N^\psi(t)=N\left(H^\psi(t)\right)$, $t>0$, where
 $N(t)$  is the Poisson process with intensity $\lambda $  and $H^\psi(t)$ is the subordinator with  Bern\v{s}tein function $\psi(u)$, independent of $N(t)$. This class of processes was introduced and studied in  \cite{OT} and called by the authors  Poisson processes with Bern\v{s}tein intertimes.  
  It was shown in  \cite{OT} that the distributions of $N^\psi(t)$, $t>0$, can be presented as follows:
 \begin{equation}\label{pkf}
            p_k^\psi(t)=\Prob\left\{N^\psi(t)=k\right\}=\frac{(-1)^k}{k!}\frac{d^k}{dt^k}e^{-t\psi(\lambda u)}\arrowvert_{u=1},
 \end{equation}
 and satisfy the difference-differential equations:
 \begin{equation}
            \frac{d}{dt}p_k^\psi(t)=-\psi(\lambda)p_k^\psi(t)+\sum_{m=1}^{k}\frac{\lambda^m}{m!}p_{k-m}^\psi(t)\int_{0}^{\infty}e^{-s \lambda}s^m \nu(ds), \,  k\geq0, \, t>0,\nonumber
 \end{equation}
        which can be also written as
 \begin{equation}\label{dtpkf}
            \frac{d}{dt}p_k^\psi(t)=-\psi(\lambda)p_k^\psi(t)-\sum_{m=1}^{k}\psi^{(m)}(\lambda)\frac{(-\lambda)^m}{m!}p_{k-m}^\psi(t), \,  k\geq0, \, t>0,
 \end{equation}
with the usual initial conditions: 
  $p_0^\psi(0)=1$,  $p_k^\psi(0)=0$ for $k\ge 1$. The last equation can be represented in the form:
  \begin{equation}\label{dtpkf1}
            \frac{d}{dt}p_k^\psi(t)=-\psi(\lambda(I-B))p_k^\psi(t), \, k\geq0, \, t>0,
  \end{equation}
            where $B$ is the backshift operator: $B p_k^f(t)=p_{k-1}^f(t)$, and it is supposed that $p_{-1}(t)=0$. 
 
 The probability generating function of the process $N^\psi$ is of the form (\cite{OT}):
\begin{equation}\label{G23}
             G^\psi(u,t)=\E u^{N^\psi(t)}=e^{-t \psi(\lambda(1-u))},\, |u|<1.
\end{equation}

            We refer for more detail on these processes to \cite{OT}.
            
            Consider the process $N^\psi$ time-changed by an independent inverse subordinator $Y^f$. We can state the following result. 
 
\begin{theorem}\label{t2}
        The process $N^{\psi, f}(t) = N\left(H^\psi(Y^f(t))\right)$, $t>0$, has probability distribution function
        \begin{equation}
        	p_{k}^{\psi,f}(t)=\Prob\left\{N^{\psi,f}(t)=k\right\}=\frac{\left(-\lambda \partial_{\lambda}\right)^k }{k!}\tilde{\ell}_f(t,\psi(\lambda)) \label{pkt1},
        \end{equation}
        and the probabilities $p_k^{\psi,f}$ satisfy the following equation
    \begin{equation}
        	\mathcal{D}^f_t p_k^{\psi,f}(t)=-\psi\left(\lambda\left(I-B\right)\right)p_k^{\psi,f}(t) \label{Dtfpkpsif}
        	    \end{equation}
        	    with the usual initial conditions.
        	    
        	    The probability generating function of the process $N^{\psi, f}$ has the form
\begin{equation}\label{Gpsif}
             G^{\psi,f}(u,t)=\tilde \ell(t, \psi(\lambda(1-u))),\, |u|<1,
\end{equation}
and satisfies the equation
\begin{equation}\label{DftGpsif}
        	\mathcal{D}^f_t G^{\psi,f}(u,t)=-\psi(\lambda(1-u))G^{\psi,f}(u,t)
    \end{equation}
    with $G^{\psi,f}(u,0)=1$.
\end{theorem}       
    
\begin{proof} 
	To prove equation \eqref{pkt1} we  perform the calculations as follows:
         \begin{eqnarray}
                p_k^{\psi,f}(t)
                &=&\Prob\left\{N\left(H^{\psi}(Y^f(t))\right)=k\right\}\nonumber\\
                 &=&\int_{0}^{\infty}\Prob\left\{N(s)=k\right\}\Prob\left\{H^{\psi}(Y^f(t))\in ds\right\}\nonumber\\
                 &=&\int_{0}^{\infty}\frac{\left(-\lambda \partial_{\lambda}\right)^k }{k!}e^{-\lambda s}\Prob\left\{H^{\psi}(Y^f(t))\in ds\right\}\nonumber\\
                 &=&\E\left[\frac{\left(-\lambda \partial_{\lambda}\right)^k }{k!}\exp\left\{-\lambda H^{\psi}(Y^f(t))\right\}\right]\nonumber\\
                 &=&\frac{\left(-\lambda \partial_{\lambda}\right)^k }{k!}\E\exp\left\{-\lambda H^{\psi}(Y^f(t))\right\}.\nonumber
         \end{eqnarray}
         Now it is left to note that
         \begin{equation}
                \E\exp\left\{-\lambda H^{\psi}(Y^f(t))\right\}=\E\exp\left(-\psi(\lambda)Y^f(t)\right)=\tilde{\ell}_f(t,\psi(\lambda)). \nonumber
         \end{equation}
         Let us state equation \eqref{Dtfpkpsif}. We have
         \begin{equation}
                p_k^{\psi,f}(t)
                =\Prob\left\{N\left(H^{\psi}(Y^f(t))\right)=k\right\}=\int_{0}^{\infty}p_k^{\psi}(u)\ell_f(t,u)du,\quad k=0,1,2,\dots \nonumber
         \end{equation}
         We repeat the same lines as in the proof of Theorem \ref{t1}. We take the generalized R-L derivative $\D_t^f$ and use equations \eqref{dtpkf}--\eqref{dtpkf1}:
         \begin {eqnarray}\label{dtfpkpsif}
                \D_t^fp_k^{\psi,f}(t)
                &=&\int_{0}^{\infty}p_k^{\psi}(u)\D_t^f\ell_f(t,u)du=-\int_{0}^{\infty}p_k^{\psi}(u)\frac{d}{du}\ell_f(t,u)du\nonumber\\
                 &=&\int_{0}^{\infty}\ell_f(t,u)\frac{d}{du}p_k^{\psi}(u)du-p_k^{\psi}(u)\ell_f(t,u)\big|_{0}^{\infty}\nonumber\\
                 &=&\int_{0}^{\infty}\ell_f(t,u)\left(-\psi(\lambda(I-B))\right) p_k^\psi(u)du+p_k^{\psi}(0)\ell_f(t,0)\nonumber\\
                 &=&\int_{0}^{\infty}\ell_f(t,u)\Big(-\Big(\psi(\lambda) p_k^\psi(u)\nonumber\\
                 &&+\sum_{m=1}^{k}\psi^{(m)}(\lambda)\frac{(-\lambda)^m}{m!}p_{k-m}^\psi(u)\Big)\Big)du
                 +            p_k^{\psi}(0)\nu(t)\nonumber\\
                 &=&-\psi(\lambda) p_k^{\psi,f}(t)-\sum_{m=1}^{k}\psi^{(m)}(\lambda)\frac{(-\lambda)^m}{m!}p_{k-m}^{\psi,f}(t)+  p_k^{\psi}(0)\nu(t)\nonumber\\
                 &=&-\psi(\lambda(I-B))p_k^{\psi,f}(t)+p_k^{\psi}(0)\nu(t).
         \end{eqnarray}
         Using the relation between generalized derivatives of C-D and R-L types, we can write
         \begin{equation}
                \mathcal{D}^f_t p_k^{\psi,f}(t)=\D_t^fp_k^{\psi,f}(t)-\nu (t)p_k^{\psi,f}(0);\nonumber
         \end{equation}
         and we have (see \eqref{dlfin}):
         \begin{equation}
                p_k^{\psi,f}(0)=\int_{0}^{\infty}p_k^{\psi}(u)\ell_f(0,u)du=\int_{0}^{\infty}p_k^{\psi}(u)\delta(u)du=p_k^{\psi}(0)=1.\nonumber
         \end{equation}
         Therefore, we obtain
         \begin{equation}
                \mathcal{D}^f_t p_k^{\psi,f}(t)=-\psi(\lambda(I-B))p_k^{\psi,f}(t).\nonumber
         \end{equation}
          Using the expression for $G^{\psi}(u,t)$ given by \eqref{G23}, we calculate $G^{\psi,f}:$
	 \begin{eqnarray*}
		 G^{\psi,f}(u,t)&=&\E u^{N\left(H^{\psi}(Y^f(t))\right)} =\int_{0}^{\infty}G^{\psi}(u,s)\ell_f(t,s)ds= \\
			  	  &=&  \int_{0}^{\infty}e^{-s\psi (\lambda(1-u))}\ell_f(t,s)ds=\E e^{-\psi(\lambda(1-u))Y^f(t)}=\\
			             &=& \tilde{\ell}_f(t, \psi (\lambda(1-u))),
	 \end{eqnarray*}
	and since $\tilde{\ell}_f(t, \psi (\lambda(1-u)))$ is an eigenfunction of $\mathcal{D}^f_t$ with the eigenvalue $\psi (\lambda(1-u))$ (see \eqref{lapl}), it follows that $G^{\psi, f}(u,t)$ satisfies equation \eqref{DftGpsif}.    
         \end{proof}
\begin{remark}{\rm 
We notice at once that  equations \eqref{Dtfpkpsif} for the probabilities of the process $N^{\psi, f}(t) = N(H^\psi(Y^f(t)))$ mimic the corresponding equations for the process $N^{\psi}(t) = N(H^\psi(t))$, that is, the process before the time change by an inverse subordinator, only the ordinary derivative in time  is changed for the generalized fractional derivative.  This can be anticipated quite straightforwardly,  in view of the technique used for the proof, and holds also for other models, like those below and in the next section. 
On the other side, the equations \eqref{Dtfpkpsif} tell us, that for the probabilities of the processes with double time change like $N\left(H^\psi(Y^f(t))\right)$, the action (in time) of the operator $\mathcal{D}^{f}_t$ (which is related to $Y^f$) is equal to the action (in space) of the operator  $-\psi\left(\lambda\left(I-B\right)\right)$ (which is related to $H^\psi$, and depends also of the outer process $N$).
Some more insight on these operators can be found within the approach applied in the paper \cite{BDOV}.  Let $Y^f=Y^\beta$ be the inverse stable subordinator with a parameter $\beta$, that is, $f(\lambda)=\lambda^\beta$, then equation \eqref{Dtfpkpsif} becomes
\begin{equation}
        	\mathcal{D}^{\beta}_t p_k^{\psi,\beta}(t)=-\psi\left(\lambda\left(I-B\right)\right)p_k^{\psi,\beta}(t),
        	 \label{Dtfpkpsistable}
        	    \end{equation}
        	    where $\mathcal{D}^{\beta}_t$ is the Caputo-Djrbashian fractional derivative \eqref{C-D}. If we suppose furthermore that $H^\psi=H^\alpha$ is the stable subordinator with a parameter $\alpha$, that is, $\psi(\lambda)=\lambda^\alpha$, then we come to the  equation 
        	    \begin{equation}
        	\mathcal{D}^{\beta}_t p_k^{\alpha,\beta}(t)=-\lambda^\alpha\left(I-B\right)^\alpha p_k^{\alpha,\beta}(t).
        	 \label{Dtfpkstablestable}
        	    \end{equation}
Equations \eqref{Dtfpkpsistable} and \eqref {Dtfpkstablestable} were stated in paper \cite{BDOV}, where the particular representation of the operator  $\psi\left(\lambda\left(I-B\right)\right)$ was used, and the equations for the probabilities of time changed processes were stated using the interplay between the operator $\psi\left(\lambda\left(I-B\right)\right)$ and the fractional operator  $\mathcal{D}^{\beta}_t$ in the Fourier domain. The approach in \cite{BDOV} can be extended for  the case where we deal with the generalized fractional operator $\mathcal{D}^{f}_t$, as soon as we know that its eigenfunction is given by $\tilde{\ell}_f(t,\lambda)$ (see \eqref{lapl}). We refer for more detail on this approach to \cite{BDOV}. 
}	
\end{remark}

Consider a generalization of counting processes with Bern\v{s}tein intertimes, which was introduced in the paper \cite{GOS}:
\begin{equation}
     N^{\psi_1, \psi_2, \dots,\psi_n}(t)=N\left(\sum_{j=1 }^{n}H^{\psi_j}(t)\right),t\geq0 ,\nonumber
\end{equation}
where $H^{\psi_i}, i=1,\dots n$ are $n$ independent subordinators with Bern\v{s}tein functions $\psi_i$, independent of the Poisson process $N$. In \cite{GOS} the following result was stated.
\begin{prop}[\cite{GOS}]\label{p1}
    The distribution of the subordinated process $N^{\psi_1, \psi_2, \dots,\psi_n}(t),$ $t\geq 0$, is the solution to the Cauchy problem
    \begin{equation*}
         \frac{d}{dt}p_k^{\psi_1, \psi_2, \dots,\psi_n}(t)=-\sum_{j=1 }^{n}\psi_j(\lambda(I-B))p_k^{\psi_1, \psi_2, \dots,\psi_n}(t) ,
    \end{equation*}
    with the usual initial conditions
    \begin{equation*}
         p_k^{\psi_1, \psi_2, \dots,\psi_n}(0)=
         \begin{cases}
                1,\quad k>0,\\
                0,\quad k=0.\\
         \end{cases}
     \end{equation*}
\end{prop}

     Consider the process $N^{\psi_1, \psi_2, \dots,\psi_n,f}(t)=N^{\psi_1, \psi_2, \dots,\psi_n}(Y^f(t)), t\geq 0.$

\begin{theorem}\label{th3}
    The process $N^{\psi_1, \psi_2, \dots,\psi_n,f}$ has probability distribution function
    \begin{equation*}
      p_k^{\psi_1, \psi_2, \dots,\psi_n,f}(t)=\frac{\left(-\lambda \partial_{\lambda}\right)^k }{k!}\tilde{\ell}_f(t,\sum_{j=1 }^{n}\psi_j(\lambda)),
    \end{equation*}
    and
    the probabilities $p_k^{\psi_1, \psi_2, \dots,\psi_n,f}(t)$ satisfy the following equation
    \begin{equation*}
        	\mathcal{D}^f_t p_k^{\psi_1, \psi_2, \dots,\psi_n,f}(t)=-\sum_{j=1 }^{n}\psi_j\left(\lambda\left(I-B\right)\right)p_k^{\psi_1, \psi_2, \dots,\psi_n,f}(t) 
    \end{equation*}
    with the usual initial conditions.
\end{theorem}
Proof of Theorem \ref{th3} is obtained by the same reasoning as that of Theorem \ref{t2}, using Proposition \ref{p1}.

\bigskip
\noindent {\bf Example 1.}
        Consider the time-changed process $N^{GN}(t)=N_{1}(G_{N}(t)),$  $t>0$, where $N_{1}(t)$ is the Poisson process with intensity $\lambda _{1},$ and $G_N(t)=G(N(t))$, $t>0$, is the compound Poisson-Gamma subordinator with parameters $\lambda, \alpha, \beta$, that is, with the Laplace exponent $\psi_{GN}(u)=\lambda \beta^\alpha \left(\beta^{-\alpha}-(\beta+u)^{-\alpha}\right).$ In the case when $\alpha=1$ we have the compound Poisson-exponential subordinator, which we will denote as $E_N(t)$, and the corresponding time-changed process as $N^{E}(t)$. The detailed study of the processes $N^{GN}(t)$ and $N^{E}(t)$ was presented in \cite{BS2017}, \cite{BS2018}.

\noindent We now consider the processes $N^{GN,f}(t)=N^{GN}(Y^f(t))$  and $N^{E,f}(t)=N^{E}(Y^f(t))$, $t\geq0.$

The next theorem is obtained as a corollary of Theorem \ref{t2}.
\begin{theorem}
        The process $N^{GN,f}(t)$ has probability distribution function
        \begin{equation*}
             p_k^{GN,f}(t)=\frac{\left(-\lambda_1 \partial_{\lambda_1}\right)^k }{k!}\tilde{\ell}_f(t,\psi_{GN}(\lambda_1)),
        \end{equation*}
        and  $p_k^{GN,f}(t)$ satisfy the following equation
    \begin{eqnarray*}
        	\mathcal{D}^f_t p_k^{GN,f}(t)
        &=&\left( \frac{\lambda \beta ^{\alpha }}{(\lambda
		_{1}+\beta )^{\alpha }}-\lambda \right) p_{k}^{GN,f}(t)\\
            &&+\frac{\lambda \beta
		^{\alpha }}{(\lambda _{1}+\beta )^{\alpha }}\sum_{m=1}^{k}\frac{\lambda
		_{1}^{m}}{\left( \lambda _{1}+\beta \right) ^{m}}\frac{\Gamma (m+\alpha )}{%
		m!\Gamma (\alpha )}p_{k-m}^{GN,f}(t).
    \end{eqnarray*}
    The process $N^{E,f}(t)=N_1(E_N(Y^f(t)))$  has probability distribution function 
    \begin{equation*}
             p_k^{E,f}(t)=\frac{\left(-\lambda_1 \partial_{\lambda_1}\right)^k }{k!}\tilde{\ell}_f(t,\psi_E(\lambda_1)),
    \end{equation*} 
    where $\psi_E(u)=\frac{\lambda u}{\beta+u},$  and $ p_k^{E,f}(t)$ satisfy the following equation 
    \begin{equation*}
        	\mathcal{D}^f_t p_k^{E,f}(t)
        =-\lambda \frac{\lambda_1}{\lambda
		_{1}+\beta }p_k^{E,f}(t)+\frac{\lambda \beta
		}{\lambda _{1}+\beta }\sum_{m=1}^{k}\left(\frac{\lambda
		_{1}}{ \lambda _{1}+\beta }\right) ^{m}p_{k-m}^{E,f}(t).
    \end{equation*}
\end{theorem}

\section{Models of time-changed generalized counting processes}
Consider the generalized counting process (GCP) $M(t)$, $t\ge 0$, which was introduced in \cite{DiC}. The probabilities $\tilde p_n(t)=\Prob\left\{M(t)=n\right\}$  depend on $k$ parameters $\lambda_1, \ldots, \lambda_k$ and are given by the formula
\begin{equation}
             \tilde p_n(t)=\sum_{\Omega(k,n)}^{}\prod_{j=1}^{k}\frac{\left(\lambda_j t\right)^{x_j}}{x_j!}e^{-\Lambda t}, n\geq0,\label{pnt}
\end{equation}
where $\Omega(k,n)=\left\{(x_1,\dots,x_k)\,:\,\sum_{j=1}^{k}jx_j=n,x_j\in N_0\right\}
$, $\Lambda=\sum_{j=1}^{k}\lambda_j$. \\
GCP performs $k$ kinds of jumps of amplitude $1, 2, \ldots, k$ with rates $\lambda_1, \ldots, \lambda_k$.

The probability generating function of GCP is given by (\cite{KK2022a}):
\begin{equation}
             \tilde G(u,t)=\E u^{M(t)}=\exp\Big\{
-\sum_{j=1}^{k}\lambda_j(1-u^j)t\Big\},\, |u|<1.
\end{equation}

The probabilities \eqref{pnt} satisfy
\begin{equation*}
            \frac{\tilde p_n(t)}{dt}=- \Lambda \tilde p_n(t) + \sum_{j=1}^{\min\{n,k\}}\lambda_j \tilde p_{n-k}(t), \, n\geq0,
\end{equation*}
with the usual initial condition. 
The probabilities $\tilde p_n(t)$ can be also written as 
\begin{equation}\label{pnt31}
             \tilde p_n(t)=\sum_{\Omega(k,n)}^{}\prod_{j=1}^{k}\frac{\lambda_j^{x_j} }{x_j!}\left(-\partial_{\Lambda}\right)^{z_k} e^{-\Lambda t}, \, n\geq0,
\end{equation}
where 
             $z_k=\sum_{j=1}^{k}x_j$ (see \cite{KK2022Arxiv}).

Recently various models  of time-changed GCP and fractional GCP were studied, in particular, with time-change given by L\'evy subordinators, inverse subordinators, including the cases of some specific subordinators, and also the fractional extensions of processes  of the form    $M(H^{\psi}(t))$   (see, for example, \cite{DiC}, \cite{KLS}, \cite{KK2022a}, \cite{KK2022b}, \cite{KK2022Arxiv},  and references therein).
  
  Following the lines of the previos section, we consider the time-changed process  $\mathsf{M}^{\psi,f}(t)=M(H^{\psi}(Y^f(t))$, that is, with double time-change by an independent subordinator $H^{\psi}$ and an inverse subordinator $Y^f$, which are independent of $M$. To the best of our knowledge, such general case has not been presented in the literature. In the next theorem we characterize its probabilities $\tilde{p}_n^{\psi,f}(t)=\Prob\left\{\mathsf{M}^{\psi,f}(t)=n\right\}$ and probability generating function.
\begin{theorem}\label{Th5}
   	 The process $\mathsf{M}^{\psi,f}$ has probability distribution function
   	 \begin{equation}\label{pnpsif}
   	          \tilde{p}_n^{\psi,f}(t)=\sum_{\Omega(k,n)}^{}\prod_{j=1}^{k}\frac{\lambda_j ^{x_j}}{x_j!}\left(-\partial_{\Lambda}\right)^{z_k}\tilde{\ell}_f(t,\psi (\Lambda)), \, n\ge 0,
    	\end{equation}
    	and $\tilde{p}_n^{\psi,f}(t)$ satisfy the following equation
    	\begin{eqnarray}
   		 \mathcal{D}^f_t \tilde{p}_n^{\psi,f}(t)=-\psi(\Lambda)\tilde{p}_n^{\psi,f}(t) -\sum_{m=1}^{n}\sum_{\Omega(k,m)}^{}\psi^{(z_k)}\left(\Lambda\right) \prod_{j=1}^{k}\frac{(-\lambda_j )^{x_j}}{x_j!}\tilde{p}_{n-m}^{\psi,f}(t),\label{Dtfpn}
  	\end{eqnarray}
   	 with initial conditions 
    \begin{equation*}
         \tilde{p}_n^{\psi,f}(0)=
         \begin{cases}
                1,\quad n=0,\\
                0,\quad n\geq1.\\
         \end{cases}
     \end{equation*}
        The probability generating function of the process $\mathsf{M}^{\psi,f}$ is of the form
\begin{equation}\label{Gpsif1}
             \tilde G^{\psi,f}(u,t)=\tilde \ell_f\Big(t, \psi\Big(\sum_{j=1}^{k}\lambda_j(1-u^j)\Big)\Big),\, |u|<1,
\end{equation}
and satisfies the equation
\begin{equation}\label{DftG}
        	\mathcal{D}^f_t \tilde G^{\psi,f}(u,t)=-\psi\Big(\sum_{j=1}^{k}\lambda_j(1-u^j)\Big)\tilde G^{\psi,f}(u,t)
    \end{equation}
    with $\tilde G^{\psi,f}(u,0)=1$.
\end{theorem}  

\begin{proof}
For calculating $\tilde{p}_n^{\psi, f}(t)$ we use the formula \eqref{pnt31}:
        \begin{eqnarray*}
                 \tilde{p}_n^{\psi, f}(t) &=& \int_{0 }^{\infty}\tilde{p}_n(s)P\left\{H^{\psi}(Y^f(t)) \in ds\right\}= \\
                 &=& \int_{0 }^{\infty}\sum_{\Omega(k,n)}\prod_{j=1}^{k}\frac{\lambda_j ^{x_j}}{x_j!}\left(-\partial _{\Lambda}\right)^{z_k}e^{-\Lambda t}P\left\{H^{\psi}(Y^f(t))\in ds\right\}=\\
                 &=& \sum_{\Omega(k,n)}\prod_{j=1}^{k}\frac{\lambda_j ^{x_j}}{x_j!}\E\left(-\partial _{\Lambda}\right)^{z_k}\exp\left\{-\Lambda H^{\psi}(Y^f(t))\right\}=\\
                 &=&\sum_{\Omega(k,n)}\prod_{j=1}^{k}\frac{\lambda_j ^{x_j}}{x_j!}\left(-\partial _{\Lambda}\right)^{z_k}\E \exp\left\{-\Lambda H^{\psi}(Y^f(t))\right\}.
        \end{eqnarray*}
        This implies formula \eqref{pnpsif}, since 
        \begin{equation*}
          \E \exp\left\{-\Lambda H^{\psi}(Y^f(t))\right\}=\E \exp\left\{-\psi(\Lambda) Y^f(t)\right\}=\tilde{\ell}_f(t,\psi (\Lambda)).
        \end{equation*}
        To derive equation \eqref{Dtfpn} we follow the same lines as is the proof of Theorem \ref{t2}, but we use now instead of the equation \eqref{dtpkf} (or \eqref{dtpkf1}) for probabilites of the process $N^{\psi}=N(H^{\psi}),$ the following equation for the probabilities $\tilde{p}_n^{\psi}(t)$ of the process $M^{\psi}=M(H^{\psi})$(see \cite{KK2022Arxiv}):
        \begin{equation*}
          \frac{d}{dt}\tilde{p}_n^{\psi}(t)= -\psi (\Lambda)\tilde{p}_n^{\psi}(t)-\sum_{m=1}^{n}\sum_{\Omega(k,m)}\psi^{z_k} (\Lambda)\prod_{j=1}^{k}\frac{(-\lambda_j)^{x_j}}{x_j!}\tilde{p}_{n-m}^{\psi}(t).
        \end{equation*}
              
        Thus, we obtain 
        \begin{equation*}
          {\mathbb{D}}_t^f\tilde{p}_n^{\psi,f}(t)=\psi(\Lambda)\tilde{p}_n^{\psi,f}(t)-\sum_{m=1}^{n}\sum_{\Omega(k,m)}\psi^{z_k} (\Lambda)\prod_{j=1}^{k}\frac{(-\lambda_j)^{x_j}}{x_j!}\tilde{p}_{n-m}^{\psi,f}(t)+\tilde{p}_{n}^{\psi,f}(0)\nu(t),
        \end{equation*}
        and then apply the same reasonings as those after formula \eqref{dtfpkpsif} to come to \eqref{Dtfpn}.

        Using the expression for the probability generating function $\tilde{G}^{\psi}(u,t)$ of the time-changed process $M^{\psi}(t)=M(H^{\psi}(t))$ (see  \cite{KK2022a}):
        \begin{equation*}
          \tilde{G}^{\psi}(u,t)=\exp\Big\{-t\psi \Big(\sum_{j=1}^{k}\lambda_j(1-u^j)\Big)\Big\},
        \end{equation*}
        we calculate $\tilde{G}^{\psi,f}$ as follows:
        \begin{eqnarray*}
        % \nonumber % Remove numbering (before each equation)
          \tilde{G}^{\psi,f}(u,t) &=& \int_{0}^{\infty}\tilde{G}^{\psi}(u,s)\ell_f(t,s)ds=\\
           &=&\int_{0}^{\infty}\exp\Big\{-t\psi \Big(\sum_{j=1}^{k}\lambda_j(1-u^j)\Big)\Big\}\ell_f(t,s)ds=\\
          &=& \E e^{-\psi \left(\sum_{j=1}^{k}\lambda_j(1-u^j)\right)Y^f(t)}=\tilde{\ell}_f\Big(t,\psi\Big(\sum_{j=1}^{k}\lambda_j(1-u^j)\Big)\Big), 
        \end{eqnarray*}
        that is, \eqref{Gpsif1} holds. Equation \eqref{DftG} follows in view of  \eqref{Gpsif1} and \eqref{lapl}.
\end{proof}

\begin{remark}   {\rm Following the same calculation as in \cite[formula (5.3)]{KK2022Arxiv}, equation \eqref{Dtfpn} can also be written in the form
     \begin{equation*}
        \mathcal{D}^f_t \tilde{p}_n^{\psi,f}(t)
        =-\psi\Big(\Lambda\Big(I-\frac{1}{\Lambda}\sum_{j=1}^{n\wedge k}\lambda_jB^j\Big)\Big) \tilde{p}_{n-m}^{\psi,f}(t).
    \end{equation*}
    }    \end{remark}

\noindent {\bf Example 2.} As a continuation of Example 1, consider GCP with time-change given  by a compound Poisson-Gamma subordinator: $M^{GN}(t)=M(G_N(t))$. Probabilities of this process can be calculated and expressed in terms of special functions:
\begin{eqnarray*}
	\tilde p_{n}^{GN}(t) &=&\Prob\left\{M(G_{N}(t))=n\right\}=\int_{0}^{\infty
	}\Prob(M(s)=n)\Prob(G_{N}(t)\in ds) \\
	&=&\int_{0}^{\infty }\sum_{\Omega(k,n)}^{}\prod_{j=1}^{k} \frac{\left( \lambda _{j}s\right) ^{x_j}%
	}{x_j!}e^{-\Lambda s}
	\\
	&&\times \Big\{ e^{-\lambda t}\delta _{\{0\}}(ds)+e^{-\lambda t-\beta s}\frac{1%
	}{s}\Phi \left( \alpha ,0,\lambda t(\beta s)^{\alpha }\right) \Big\} ds \\
	&=&e^{-\lambda t}\int_{0}^{\infty }\sum_{\Omega(k,n)}^{}\prod_{j=1}^{k} \frac{\left( \lambda _{j}s\right) ^{x_j}%
	}{x_j!}e^{-\Lambda s}e^{-\beta s} \frac{1}{s}\Phi \left( \alpha ,0,\lambda t(\beta
	s)^{\alpha }\right) ds \\
	&=&e^{-\lambda t}\sum_{\Omega(k,n)}^{}\int_{0}^{\infty }\prod_{j=1}^{k} \frac{\left( \lambda _{j}s\right) ^{x_j}%
	}{x_j!} e^{-(\Lambda+\beta)s}\frac{1}{s}\sum_{l=1}^{\infty}\frac{(\lambda t \beta^{\alpha})^l}{l!\Gamma(\alpha l)}s^{\alpha l-1}ds\\
    &=&e^{-\lambda t}\sum_{\Omega(k,n)}^{}\prod_{j=1}^{k} \frac{ \lambda _{j} ^{x_j}%
    	}{x_j!} \sum_{l=1}^{\infty}\frac{(\lambda t \beta^{\alpha})^l}{l!\Gamma(\alpha l)}\int_{0}^{\infty }e^{-(\Lambda+\beta)s}s^{\sum_{j=1}^{k}x_j+\alpha l-1}ds\\
    &=&e^{-\lambda t}\sum_{\Omega(k,n)}^{}\prod_{j=1}^{k} \frac{ \lambda _{j} ^{x_j}%
    	}{x_j!} \sum_{l=1}^{\infty}\frac{(\lambda t \beta^{\alpha})^l}{l!\Gamma(\alpha l)}\frac{\Gamma\left(z_k+\alpha l\right)}{\left(\Lambda+\beta\right)^{z_k+\alpha l}}\\
    &=&e^{-\lambda t}\sum_{\Omega(k,n)}^{}\prod_{j=1}^{k} \frac{ \lambda _{j} ^{x_j}%
    	}{x_j!}\frac{1}{\left(\Lambda+\beta\right)^{z_k}}     \sum_{l=1}^{\infty}\frac{(\lambda t \beta^{\alpha})^l}{l! \left(\Lambda+\beta\right)^{\alpha l}} \frac{\Gamma\left(z_k+\alpha l\right)}{\Gamma(\alpha l)},
    	\end{eqnarray*}
    	which can be written in the form
 \begin{equation*}
\tilde p_{n}^{GN}(t)=e^{-\lambda t}\sum_{\Omega(k,n)}^{}\prod_{j=1}^{k} \frac{ \lambda _{j} ^{x_j}%
	}{x_j!}\frac{1}{\left(\Lambda+\beta\right)^{z_k}}\, {_1}\Psi_1\left((z_k,\alpha), (0,\alpha), \frac{\lambda t \beta^{\alpha}}{\left(\Lambda+\beta\right)^{\alpha }} \right),
\end{equation*}
where ${_p}\Psi_q$ is the generalized Wright function 
	\begin{equation*}
	_{p}\Psi_{q}((a_i, \alpha_i), (b_j, \beta_j), z)=\sum_{k=0}^{\infty
	}\frac{\prod\limits_{i = 1}^{p}\Gamma (a_i+\alpha_i
		k)}{\prod\limits_{j = 1}^{q}\Gamma (b_j+\beta_j k)} \frac{z^k}{k!}
	\end{equation*}
	 defined for $z\in C$, $a_i, b_i\in C$, $\alpha_i, \beta_i\in R$, $\alpha_i, \beta_i\neq 0$ and $\sum \alpha_i - \sum \beta_i > -1$ (see, e.g., \cite{HMS}).

If $\alpha=1$, that is, $G_N(t)=E_N(t)$, then probabilities become of the form
\begin{equation*}
	\tilde p_{n}^{E}(t)=e^{-\lambda t}\sum_{\Omega(k,n)}^{}\prod_{j=1}^{k} \frac{ \lambda _{j} ^{x_j} }{x_j!}\frac{1}{\left(\Lambda+\beta\right)^{z_k}}\mathcal{E}_{1,2}^{z_k+1} \left(\frac{\lambda t \beta}{\Lambda+\beta} \right),
\end{equation*}
where $\mathcal{E}_{\rho,\delta}^{\gamma}$ is the three-parameter generalized Mittag-Leffler function, which is  defined as follows:
\begin{equation*}
\mathcal{E}_{\rho ,\delta }^{\gamma }(z)=\sum_{k=0}^{\infty }\frac{\Gamma
	(\gamma +k)}{\Gamma (\gamma )}\frac{z^{k}}{k!\Gamma (\rho k+\delta )},\text{
}z\in C,\text{ }\rho ,\delta ,\gamma \in C,  \label{ML3}
\end{equation*}
with $\text{ Re}(\rho )>0,\text{Re}(\delta )>0,\text{Re}(\gamma )>0$ (see, e.g., \cite{HMS}).
For $n=0$ we have:
\begin{equation*}
        	 \tilde p_0^{GN}(t)
        =e^{-\lambda t}+e^{-\lambda t}\sum_{l=1}^{\infty}\frac{\left( \lambda t\beta^{\alpha } \right) ^{l}}{\left( \Lambda +\beta \right) ^{\alpha l}}\frac{\Gamma (\alpha l )}{%
		l!\Gamma (\alpha l)}
        =\exp\Big\{-\lambda t \Big(1-\frac{\beta^{\alpha}}{\left(\Lambda+\beta\right)^{\alpha}}\Big)\Big\},
    \end{equation*}
    and if $\alpha=1$
    \begin{equation*}
        	 \tilde p_0^{E}(t)=\exp\Big\{-\frac{\lambda\Lambda t}{\Lambda+\beta}\Big\}.
    \end{equation*}
   Consider now the process $M^{GN,f}(t)=M(G_N(Y^f(t))$, $t\ge0$. The next theorem follows from Theorem \ref{Th5}.
    \begin{theorem}\label{t6}
          The process $M^{GN,f}$ has probability distribution function given by formula \eqref{pnpsif} with $\psi=\psi_{GN}$ and the probabilities $\tilde{p}_n^{GN,f}(t)$ satisfy the following equation:
          \begin{eqnarray*}
            {\mathcal{D}}_t^f\tilde{p}_n^{GN,f}(t) &=& \left(\frac{\lambda \beta^{\alpha}}{\left(\Lambda+\beta\right)^{\alpha}}-\lambda\right)\tilde{p}_n^{GN,f}(t) +\\
                                        && +\frac{\left(\lambda\beta\right)^{\alpha}}{\left(\Lambda+\beta\right)^{\alpha}\Gamma(\alpha)}\sum_{m=1}^{n}\sum_{\Omega(k,m)}\frac{\Gamma\left(z_k+\alpha\right)}{\left(\Lambda+\beta\right)^{z_k}} \prod_{j=1}^{k}\frac{\lambda_j ^{x_j}}{x_j!}\tilde{p}_{n-m}^{GN,f}(t).
          \end{eqnarray*} 
            The process $M^{E,f}$has probability distribution function given by \eqref{pnpsif} with $\psi=\psi_{EN}$ and the probabilities $\tilde{p}_{n}^{E,f}(t)$ satisfy the equation
            \begin{eqnarray*}
                   {\mathcal{D}}_t^f\tilde{p}_n^{E,f}(t) &=& -\lambda\frac{\Lambda}{\Lambda+\beta}\tilde{p}_n^{E,f}(t) +\\
                                        && +\frac{\lambda\beta}{\Lambda+\beta}\sum_{m=1}^{n}\sum_{\Omega(k,m)}\frac{\Gamma\left(z_k+1\right)}{\left(\Lambda+\beta\right)^{z_k}}\prod_{j=1}^{k}\frac{\lambda_j ^{x_j}}{x_j!}\tilde{p}_{n-m}^{E,f}(t).
             \end{eqnarray*} 
    \end{theorem}

{\small
\baselineskip=0.9\baselineskip

}
\end{document}